\documentclass[12pt]{article}

\usepackage[bookmarks,bookmarksnumbered,%
    citebordercolor={0.8 0.8 0.8},filebordercolor={0.8 0.8 0.8},%
    linkbordercolor={0.8 0.8 0.8},%
    urlbordercolor={0.8 0.8 0.8},%
    pagebackref,linktocpage%
    ]{hyperref}

\usepackage[margin=1.0in]{geometry}
\usepackage{graphicx}              
\usepackage{amsmath}               
\usepackage{amsfonts}              
\usepackage{amsthm}                
\usepackage{url}
\usepackage{amssymb}
\usepackage{latexsym}
\usepackage{bm}
\usepackage{enumitem}

\allowdisplaybreaks

\newtheorem{theorem}{Theorem}[section]
\newtheorem{thm}[theorem]{Theorem}
\newtheorem{prop}[theorem]{Proposition}
\newtheorem{lemma}[theorem]{Lemma}
\newtheorem{definition}[theorem]{Definition}

\newtheorem{remark}{Remark}[section]

\theoremstyle{definition}

\let\oldmarginpar\marginpar
\renewcommand\marginpar[1]{\-\oldmarginpar[\raggedleft\footnotesize #1]%
{\raggedright\footnotesize #1}}

\def\keywords#1{\bigskip \par\noindent{\it Keywords and phrases: }#1\par}
\def\AMS#1{\par\noindent{\it 2010 Mathematics Subject Classification: }#1\par}

\DeclareMathOperator{\R}{\mathbb{R}}
\DeclareMathOperator{\C}{\mathbb{C}}
\DeclareMathOperator{\T}{\mathbb{T}}
\DeclareMathOperator{\N}{\mathbb{N}}
\DeclareMathOperator{\Z}{\mathbb{Z}}

\DeclareMathOperator{\eps}{\epsilon}

\DeclareMathOperator{\supp}{\text{supp}}
\DeclareMathOperator{\ra}{\rightarrow}
\DeclareMathOperator{\half}{\frac{1}{2}}

\newcommand{\items}{\ref{itm:a} and \ref{itm:b} }
\newcommand{\D}[3]{\widehat{\triangle^{#1}#2}(0;#3)}

\def \seq#1#2{#1_1,\dots,#1_#2}

\newcommand{\Ss}{\section}

\newcommand{\GSM}{\cite{M1}}

\newcommand{\identitiesA}{\ref{thm:the following identities}}

\title{Long progressions in sets of fractional dimension}
\author{Marc Carnovale}
\date{\today}
\begin{document}

\maketitle


\begin{abstract}
We demonstrate $k+1$-term arithmetic progressions in certain subsets of the real line whose ``higher-order Fourier dimension'' is sufficiently close to 1.
This Fourier dimension, introduced in previous work, is a higher-order (in the sense of Additive Combinatorics and uniformity norms) extension of the Fourier dimension of
Geometric Measure Theory, and can be understood as asking that the uniformity norm of a measure, restricted to a given scale, decay as the scale increases.
We further obtain quantitative information about the size and $L^p$ regularity of the set of common distances of the 
artihmetic progressions contained in the subsets of $\R$ under consideration.
\end{abstract}

\tableofcontents

\section*{Acknowledgments}

Many thanks to Prof.'s Izabella Laba and Malabika Pramanik for introducing me to the question of arithmetic progressions in fractional sets which motivated this work, 
for their patience and many suggestions with innumerable different drafts and pre-drafts of an earlier version of this material,
for the sharing of their expertise in Harmonic Analysis, for funding my master's degree, and much more. 

Thanks to Nishant Chandgotia for his forgiveness of the mess I made of our office for two years and his bountiful friendship.

Thanks to Ed Kroc, Vince Chan, and Kyle Hambrook for the frequent use of their time and ears and their ubiquitous encouragement.


\noindent \keywords{Gowers norms, Uniformity norms, singular measures, finite point configurations, Salem sets, Hausdorff dimension, Fourier dimension}
\vskip0.2in

\noindent \AMS{28A78, 42A32, 42A38, 42A45, 11B25, 28A75}


\section{Introduction}\label{ch:Introduction}

There has been interest recently in which geometric patterns may be discovered in sparse subsets of Euclidean space. In \cite{Keleti}, subsets of the real line of full Hausdorff dimension not containing any 3-term progressions, nor any boxes, were constructed; in \cite{Maga}, similar results were obtained for $\R^n$. On the other hand, in \cite{Laba}, it was shown via a Fourier restriction theorem that all sets of sufficiently high Fourier dimension contain 3-term progressions.
This is a kind of continuous analogue of Roth's theorem, and the question of higher term arithmetic progressions immediately presents itself. 
The question is of further interest owing to the suspicion that methods related to progressions might play the role of a substitute for curvature in certain parts of harmonic analysis. 
In particular, both differentiation theorems and restriction theorems, amongst others, have been found to rely on the curvature of the underlying space in essential ways, 
yet both have seen partial extensions to the fractal setting (\cite{Labadiff},\cite{Mitsis},\cite{Mocken}). 
Though we do not address the question of what might be said about restriction theorems here, after our study of $k$-term progressions in sparse subsets of the line, 
we turn our attention to the proof of a differentiation theorem. Another reason this question is important is
because of its relationship to the Falconer Distance Conjecture, which asks for measure of the set of distances between points in a set $E\subset\R^d$ of Hausdorff dimension $\alpha\geq d/2$; 
this can be thought of as asking for the size of the set of differences in the $2$-term progressions in $E$, and we answer the generalization of this question to $k+1$-term progressions
under a type of Fourier-decay condition.

In this paper, we establish sufficient conditions for a singular set in $\R$ to possess $k+1$-point configurations $(b_0,\dots,b_k)$,
by which we mean scaled and translated images of the collection of points $\{b_0,\dots,b_k\}\subset\R$.

In particular, we give conditions guaranteeing the presence $k+1$-term arithmetic progressions inside subsets of $\R$ of dimension strictly less than $1$.

Our contributions are stated in terms of a $k+1$-dimensional measure $\triangle^k\mu$ built from a measure $\mu$ on $\R$ related to Gowers' uniformity norms 
from Additive Combinatorics.

We review briefly the necessary background in Subsection \ref{ch:review}.

\begin{thm}\label{thm:existence of progressions}
Suppose that $\mu$ is a measure on $\T$ satisfying \items  with $d=1$ and  $\alpha$ and $\beta$ sufficiently close to $1$. 

Then for any $k+1$ distinct $b_0,\dots,b_{k}\in\R$, 
$\supp(\mu)$ contains an affine image $a+rb_1,\dots,a+rb_k\in\supp(\mu)$, $a,r\in\R$. Further, for any $k$ such points
$b_0,\dots,b_{k-1}$, the collection of scaling parameters $r$ for which $a+rb_0,\dots,a+rb_k\in\supp(\mu)$ is of positive Lebesgue measure.
 
\end{thm}
\begin{remark}
 The second portion of Theore \ref{thm:existence of progressions} referring to positive Lebesgue measure generalizes the easy fact that the Falconer distance conjecture is trivially true
 when the Hausdorff dimension is replaced by Fourier dimension.
\end{remark}

 The hypothesis (\ref{itm:b}) is related  to what we have termed the ``higher-order Fourier dimension'' of the measure $\mu$. In standard terminology, 
 the Fourier dimension of a measure on $\R^d$ is given by the supremum over all  $d>\beta>0$ for which $|\widehat{\mu}(\xi)|\lesssim |\xi|^{-\frac{\beta}{2}}$ at infinity. 
 If $d\mu = f\,dx$ is absolutely continuous on $\R^d$, 
 then for $k>1$ we say that it's $k$-th order Fourier dimension is similarly the supremum over all $\beta\in(0,d)$ for which the asymptotic decay rate of
 $|\widehat{\triangle^{k} f}(0;\bm{\eta})|$ is at least $|\bm{\eta}|^{-(k+1)\frac{\beta}{2}}$, where moduli some complex conjugation the function $\triangle^k f :\R^{k+1}\ra\R$ is given by 
\begin{align}\label{triangle}
\triangle^k f(x;\bm{u}) := \prod_{\iota\in\left\{0,1\right\}^k} f(x+\iota\cdot\bm{u})
\end{align}

Specializing to $b_i=i$, we obtain a sufficient condition of $k+1$-term progressions in the support of a singular measure on $\R$, 
extending the result of \cite{Laba} for $3$-term progressions.


\subsection{Review of previous work}\label{ch:review}

This paper continues work begun in \cite{M1}. There, for any measure $\mu$ on the $d$-dimensional torus $\T^d$ we introduced the $(k+1)d$-dimensional measure $\triangle^k\mu$, 
a singular analogue of the object $\triangle^k f$ relevant in the definition of Gowers norms from additive combinatorics.

We defined the $U^k$ norm of $\mu$, showed it to be equivalent to
\begin{align*}
 \|\mu\|_{U^k} = \triangle^k\mu(\T^{k+1})^{\frac{1}{2^k}}
\end{align*}
and showed that this does indeed define a norm.

%
%
%
%

Define $U^{k+1}$ to be the space of all finite measures $\mu$ on $\mathbb{T}^d$ for which $\|\left|\mu\right|\|_{U^{k+1}} < \infty$. 
Then the first part of the following theorem is a rephrasing
of part of Theorem 2 from \cite{M1}, while the second part is a portion of Proposition 1 from \cite{M1}.

	    \begin{thm}\label{thm:the following identities}
		    Let $\mu$ be a measure on $\mathbb{T}^d$. Then for all $k$, the finite measure $\triangle^{k+1}\mu$ exists if and only if $|\mu|\in U^{k+1}$.

		    Further, 
		    \begin{align*}
		      &\|\mu\|_{U^{k+1}}^{2^{k+1}}= \sum_{\bm{\eta}\in\Z^k} |\widehat{\triangle^k\mu}(0;\bm{\eta})|^2
		    \end{align*}
		    
		    \end{thm}

For definitions, you should refer to \cite{M1}.

In \cite{M2}, we introduced the following refinement of the $U^{k+1}$ norm and extension of the notion of Fourier dimension to this higher-order setting

\vspace{5pt}
  \hrulefill
\vspace{2pt}
\begin{definition}
 \emph{For $k>1$, we define the $k$th-order Fourier dimension of a measure $\mu$ on $\R^d$ to be the supremum over all $\beta\in(0,d)$ for which}
 \begin{align*}
|\widehat{\triangle^i\mu}(0;\bm{\eta})|\lesssim(1+|\bm{\eta}|)^{-\frac{i+1}{2}\beta}
 \end{align*}
\emph{ for all $i\leq k$.}
 
\emph{If $\mu$ is a measure with nontrivial compact support on $\T^d$, then we identify it with a measure on $\R^d$ in the natural way in order to define its higher-order Fourier dimension.}

\emph{We further say that the measure $\mu$ possesses a $k$th order Fourier decay of $\beta$ if for all $i\leq k$,}
\begin{align*}
|\widehat{\triangle^i\mu}(0;\bm{\eta})|\leq C(1+|\bm{\eta}|)^{-(i+1)\frac{\beta}{2}}
\end{align*}
 \end{definition}
 
\vspace{2pt}
\hrulefill
\vspace{5pt}

We can now state the two main assumptions in this paper.

 \begin{enumerate}[label=(\emph{\alph*})]
\item\label{itm:a} $|\mu|(B(x,r))\leq C_1 r^{\alpha} \text{ for all } x\in\T^d$ 

\item\label{itm:b} $|\D{k}{\mu}{\eta}|\leq C_2 (1+|\eta|)^{-\frac{\beta}{2}} \text{ for all } \eta\in\Z^d$. 
\end{enumerate}

The main result of our earlier paper \cite{M2} was that higher-order Fourier dimension gives us quantitative control that the $U^k$ norm does not in the sense of the following sense.

Let $\phi_n$ be an approximate identity with Fourier transform $\widehat{\phi_n}$ essentially supported in the ball $B(0,2^{n+1})$.

Further, set $\mu_n = \phi_n^{\ast^{k}}\ast\mu$, where $\phi_n^{\ast^k}$ refers to $k$ copies of $\phi_n$ convolved together.
Then the following is equivalent to the result of \cite{M2}

\begin{prop}[Proposition 2 of \cite{M2}]\label{thm:rk}
 Let $\mu$ be a finite compactly supported (Radon) measure on $\R^d$ with a higher order Fourier decay given by 
 
\begin{align}\label{intermsof1}
 |\widehat{\triangle^j\mu}(0;\bm{\eta})|\leq C_F (1+|\bm{\eta}|)^{-(j+1)\frac{\beta}{2}}\hspace{15pt} \forall\hspace{3pt} 1\leq j \leq k
\end{align}
Then setting
 \begin{align*}
  r_k:= \bigg(\prod_{j=3}^k \left[2-\frac{{2^{3j-2}}}{2^{3j-2}-[1-\frac{(j+1)\beta}{jd}]}\right]\bigg) (2\beta-d) \end{align*}
we have the bound

\begin{align*}
 \|\mu^{n+1}-\mu^{n}\|_{U^k}\leq C 2^{-\frac{r_k}{2^k}n}
\end{align*}
where the constant depends only on the choice of $\phi_n$, and the  constant $C_F$.

\end{prop}

We will also need the following lemma from \cite{M2}.

\begin{lemma}[Lemma 4.1 of \cite{M2}]\label{thm:start}
 Let $\mu\in U^{k}$ be a signed measure and suppose that $k\geq 2$. Then for any $\bm{\eta}\in\Z^k$, 
 \begin{align}
  |\widehat{\triangle^k\mu}(0;\bm{\eta})|\leq \widehat{\triangle^k\mu}(0;0)
 \end{align}
 \end{lemma}

\Ss{Finding progressions}
Throughout this section, let $\phi:{{{{\T}}}}\ra[0,\infty)$ be an approxiate identity with the property that

\[\widehat{\phi}|_{B(0,1)} \equiv 1.\]

We set $\phi_n:=2^n\phi(2^n\cdot)$. The approximate identity $(\phi_n)_{n\in\N}$ will remain fixed throughout the rest of this section. It should not be difficult 
to obtain similar results for any choice of approximate identity, but this is not needed for our results and the choice of $(\phi_n)$ we have made is a convenient one.

For definitions of the space $U^k$, the $U^k$ norm, and the higher-order inner product $\langle f,\,g\rangle_{U^k}$, we refer the reader to \GSM. 

Let $b_0,\dots,b_k$ be $k+1$ pairwise distinct points in ${{{{\R}}}}$. 

For functions $f_i\in L^{\infty}({{{{\T}}}})$, $i=0,\dots,k$, and $g\in L^{\infty}({{{{\T}}}}\times{{{{\T}}}})$, define
\[\Lambda_{k+1}(g;f_0,\dots,f_{k}) := \int g(x,r) \prod_{i=0}^{k}f_i(x-rb_i)\,dx\,dr\]
and for $f\in L^{\infty}({{{{\T}}}})$,
\[\Lambda_{k+1}(g;f) := \Lambda_{k+1}(g;f,\dots,f).\]

\begin{proof}[Proof of Theorem \ref{thm:existence of progressions}]
 Throughout, we may suppose $\alpha$ and $\beta$ sufficiently close to $1$. By Proposition \ref{thm:existenceofcount}, the measure $\cap^{k+1}\mu$ exists. 
 By Lemma \ref{thm:support}, $\cap^{k+1}\mu$ is supported on the collection $b_0,\dots,b_{k}$. By Lemma \ref{thm:nontri}, $\cap^{k+1}$ gives zero mass to the trivial affine images
 of the $b_i$. Thus affine images of the $k+1$ point configuration
 will be shown to exist in $\supp(\mu)$ as soon as we know that the measure $\cap^{k+1}\mu$ is not itself trivial.
 
 To show that $\cap^{k+1}\mu$ is not the zero measure, we fix $\alpha_0,\beta_0\in(0,1)$ sufficiently close to $1$ that the above argument holds.

 Let  $C_H'$ be as in Lemma \ref{thm:standard} 
and let $C = c(1,100C_H')$ be as in Lemma \ref{thm:varn}. Choose $m\in\N$ be so large that
 \[|\Lambda_{k+1}(\mu)-\Lambda_{k+1}(\mu_m)|=\lim_{N\ra\infty}|\Lambda_{k+1}(\mu_N)-\Lambda_{k+1}(\mu_m)|\leq \half C\]
 whenever $\alpha>\alpha_0$ and $\beta>\beta_0$, which is possible since the sum on the right-hand-side of (\ref{cauchy}) is convergent as $N\uparrow\infty$,
 and depends on $\alpha$, $\beta$ monotonically.
 
 Then for $\alpha$ sufficiently close to $1$ that $2^{(1-\alpha)m}\leq 100$, 
 \[\Lambda_{k+1}(\phi_m\ast\mu)>c(1,100C_H) = C\]
 by Lemma \ref{thm:varn}, while
  \[|\Lambda_{k+1}(\mu)-\Lambda_{k+1}(\mu_m)|\leq \half C\]
so that $\cap^{k+1}\mu(\T\times\T^{k+1})=\Lambda_{k+1}(\mu)>\half C$.

So $\cap^{k+1}\mu$ is a nontrivial measure. Thus we have shown that there exist nontrivial affine images of $b_0,\dots,b_k$ in $\supp(\mu)$.

To complete the proof, we must show that the collection $R$ of dilation parameters yielding affine images of $b_0,\dots,b_{k-1}$ in $\supp(\mu)$ is of positive Lebesgue measure.
To do this, let $g\in L^{2^k}(\R)$.

If we again break into telescoping sums by applying Lemma \ref{thm:existenceofcount}, Lemma \ref{thm:GCSstep}, and Lemma \ref{thm:standard}, we have
\begin{align*}
&\Lambda_k(g;\mu) \lesssim  \sum_{n=0}^{\infty} |\Lambda_k(g;h_0,\dots,h_{k-2},\mu_{n+1}-\mu_n)| \\\leq&
\sum_{n=0}^{\infty} 2^{(k-1)(1-\alpha)n}\left|\langle g,\mu_{n+1}-\mu_n\rangle_{U^{k}}\right|^{\frac{1}{2^{k-1}}}\\\leq&
\sum_{n=0}^{\infty} 2^{(k-1)(1-\alpha)n}\|g\|_{U^k}\|\mu_{n+1}-\mu_n\|_{U^{k}}\\\leq&
\sum_{n=0}^{\infty} 2^{(k-1)(1-\alpha)n}\|g\|_{L^{2^k}} 2^{-{r_k(\beta)n/2^k}}\lesssim \|g\|_{L^{2^k}}
\end{align*}
where in the second to last line we have used the Gowers-Cauchy-Schwarz inequality, and in the last line we have used the bound $\|g\|_{U^k}\leq\|g\|_{L^{2^k}}$ which is a well-known 
consequence of Holder's inequality.

 Thus the projection of $\cap^k\mu$ onto the $r$-axis has an $L^{{2^{k}}'}$ density by $L^p$ duality, and in particular, the set $R$ must contain a set of positive Lebesgue measure.
\end{proof}

We will need the following standard Lemma (cf. \cite{gowers}).

\begin{lemma}\label{thm:standard}
 Let $f_1, f_2 \in L^{\infty}({{{{\T}}}})$, and $g\in L^{\infty}({{{{\T}}}}\times{{{{\T}}}})$. Then
 \[\Lambda_{k+1}(g;f_1)-\Lambda_{k+1}(g;f_2) = \underline{\overline{\sum}}_{f_1,f_2}\Lambda_{k+1}(g;h_0,\dots,h_{k})\]
 where the symbol $\underline{\overline{\sum}}_{f_1,f_2}$ refers to a sum taken over all combinations of $h_i\in\{f_1,f_2,f_1-f_2\}, i=0,\dots,k$ for which exactly 
 one of the $h_i$ equals $f_1-f_2$. By relabelling, we may assume that $h_{k} = f_1-f_2$.
 Consequently,
  \[\big|\Lambda_{k+1}(g;f_1)-\Lambda_{k+1}(g;f_2)\big| \lesssim\sup_{h_i\in\{f_1,f_2\}}\big|\Lambda_{k+1}(g;h_0,\dots,h_{k-1},f_1-f_2)\big|.\]
  
\end{lemma}

\Ss[Existence of the k-fold intersected measure]{Existence of the measure $\cap^{k+1}\mu$}

\begin{prop}\label{thm:existenceofcount}
Suppose that $\mu$ is a measure on ${{{{\T}}}}$ satisfying \items  for $\alpha$ and $\beta$ sufficiently close to $d$. 

Then the finite measure $\cap^{k+1}\mu: C({{{{\T}}}}\times{{{{\T}}}})\ra \C$ given by
\begin{align*}g\mapsto &\int g(x,r) \,\cap^{k+1}\mu(x,r):=\Lambda_{k+1}(g;\mu)
 \\&=:\lim_{n\ra\infty} \Lambda_{k+1}(g;\phi_n\ast\mu)
\end{align*}
is well-defined. 
\end{prop}
\begin{proof}

 Consider the linear functional $g\mapsto \Lambda_{k+1}(g;\mu)$ from $C({{{{\T}}}}\times{{{{\T}}}})$ to $\C$. If we can show that it is bounded and defined on
 a dense subset, then it extends to a continuous linear funcitonal on the entire space, to which the obviously bounded linear functionals
 \[g\mapsto \Lambda_{k+1}(g;\phi_n\ast\mu)\]
 then converge on a dense subset of $C({{{{\T}}}}\times{{{{\T}}}})$. From this it would follow that $g\mapsto\Lambda_{k+1}(g;\mu)$ has a continuous extension to the entire
 space, and that convergence to this functional occurs for all functions in $C({{{{\T}}}}\times{{{{\T}}}})$. By the Riesz representation theorem, we would then obtain existence
 of the measure $\cap^{k+1}\mu$. So it suffices to prove the
 following two claims
 \begin{itemize}
  \item\label{itm:1} $\lim_{n\ra\infty} \Lambda_{k+1}(g;\phi_n\ast\mu)$ exists for all $g$ in a dense subclass of $C({{{{\T}}}}\times{{{{\T}}}})$
  \item\label{itm:2} $g\mapsto\Lambda_{k+1}(g;\mu)$ is bounded where defined.
 \end{itemize}
To prove the first claim, we take as our dense subclass the collection of trigonometric polynomials. By linearity, it suffices to consider
monomials of the form $g(x,r) = e^{2\pi i (\xi_0 x+ \eta_0 r)}$ for $\xi_0,\eta_0\in\Z^d$. Fix such a $\xi_0,\eta_0$ (and thus $g$).

We show that the limit exists by showing that the sequence is Cauchy.  Thus let $\eps>0$; we show that for all large enough $m$ and $N$,
\begin{align}\label{cauchy'} \big| \Lambda_{k+1}(g;\phi_N\ast\mu)-\Lambda_{k+1}(g;\phi_m\ast\mu)\big|\leq \eps.\end{align}

Note that $\|\widehat{\triangle^k g(0,\cdot;\cdot)}\|_{\ell^1} =1$. Then by Lemma \ref{thm:twofrequencybound} with $p=1$, we have that for any $N>m\in\N$

\[ \big|\Lambda_{k+1}(g;\phi_N\ast\mu)-\Lambda_{k+1}(g;\phi_m\ast\mu)\big| \lesssim \sum_{n=m}^{N-1}2^{-\omega_k^1(\alpha,\beta)n}\leq \sum_{n=m}^{\infty}2^{-\omega_k^1(\alpha,\beta)n}.\]

Since the terms on the right hand side above are those of a geometric series, we have verified that for all sufficiently large $m$ and $N$, (\ref{cauchy'}) holds true.

We turn now to the second claim.

Since $g\equiv 1$ is a trigonometric polynomial, applying the first claim to the measure $|\mu|$, we in particular have existence of a finite number 
$M:=\Lambda_{k+1}(1;|\mu|)$. But then for any $g$,
\[ \big|\lim_{n\ra\infty} \int g(x,r) \prod_{i=0}^{k-1}\phi_n\ast\mu(x-ir)\,dx\,dr\big| \leq \|g\|_{L^{\infty}} \lim_{n\ra\infty}\Lambda_{k+1}(1;\phi_n\ast|\mu|)=M\|g\|_{L^{\infty}}\]
which demonstrates that $g\mapsto\Lambda_{k+1}(g;\mu)$ is bounded where defined, proving the second claim and completing the proof.
\end{proof}

\Ss{The main estimate}

\begin{lemma}\label{thm:singlefrequencybound}
 Let $p\in[1,2)$ and suppose that $\mu$ is a measure on ${{{{\T}}}}$ satisfying \items  for $\alpha$ and $\beta$ sufficiently close to $d$. 
 
 Then there exists a function $\omega_k^{p}(\alpha,\beta)$ which is positive and increasing such that 
 for any $n\in\N$ and any choice of $h_i\in\{\phi_n\ast\mu,\phi_{n+1}\ast\mu\}$, $i=0,\dots,k-1$, 
\[
\big| \Lambda_{k+1}(g;h_0,\dots,h_{k-1},\mu_{n+1}-\mu_n)\big| \lesssim \|\widehat{\triangle^k g}(0,\xi;\eta)\|_{\ell^1_{\xi}\ell^p_{\eta}}^{\frac{1}{2^k}}2^{-\omega_k^{p}(\alpha,\beta)n}.
\]
\end{lemma}

First we need the following lemma.

For a function  $g:{{{{\T}}}}\times{{{{\T}}}}\ra\C$, it will be convenient to set $\triangle^0g = g$,
\[\triangle^k g(x,y;u) := \triangle^{k-1}g(x-u_{k},y;u')\triangle^{k-1}g(x,y;u')\]
and 
\[\langle f, \mu\rangle_{U^k} := \int \triangle^{k-1}g(x-u_{k},y;u')\,\triangle^{k-1}\mu(x;u')\,dy.\]

\begin{lemma}\label{thm:GCSstep}
 Suppose that $f,h_0,h_1,\dots,h_{k-1}\in L^{\infty}({{{{\T}}}})$ and $g\in L^{\infty}({{{{\T}}}}\times{{{{\T}}}})$. Then 
 \[ \big|\Lambda_{k+1}(g;h_0,\dots,h_{k-1},f)\big|\leq \|h_1\|_{L^{\infty}}\cdots\|h_{k-1}\|_{L^{\infty}}\big|\langle f,g\rangle_{U^{k+1}}\big|^{\frac{1}{2^{k}}}.\]

\end{lemma}
\begin{proof}
It will be convenient to introduce the notation $\triangle^j_u f(x) :=\triangle^j f(x)$ together with the obvious modification for a function $g:{{{{\T}}}}\times{{{{\T}}}}\ra\C$.
 We claim that for any $j<k$ and any functions $f_0,\dots,f_j$, 
 
  \begin{align*}&\big| \Lambda_{j+1}(g;f_0,\dots,f_j)\big| \\\leq& \|f_0\|_{L^{\infty}}
  \left(\int \big| \Lambda_{j}(\triangle^1_u \tilde{g};\triangle^1_uf_1,\dots,\triangle^1_{ju} f_j)\big|\,du\right)^{\half}\end{align*}
 where $\tilde{g}(x,r) := g(x-r,r)$.
 
 Assuming the claim true, induction then gives that 
 \[\big| \Lambda_{k+1}(g;f_0,\dots,f_{k})\big| \leq \|f_0\|_{L^{\infty}}\cdots\|f_{k-1}\|_{L^{\infty}} \big|\int \triangle^kg(x-kr,r;u)
 \triangle^k f_k(x-r;u)\,dx\,dr\,du\big|^{\frac{1}{2^k}}\]
 since $\|\triangle^j_u f_j\|_{L^{\infty}} \leq \|f_j\|_{L^{\infty}}^{2^j}$
 
 Setting $f_k=f$ and $f_i=h_i$ for $i<k$ then proves the lemma.
 
 So it suffices to prove the claim. We have 
 
   \begin{align*}
   &\big| \Lambda_{j+1}(  g; f_0,\cdots,  f_j)\big|\\=&
      \big|\int \big[ f_0(x)\big]\big[\int   g(x,r)\prod_{i=1}^{j} f_{i}(x-ir)\,dr\big]\,dx\big|\\\leq&
      \| f_0\|_{L^2}
      \bigg|\int \big|\int   g(x,r)\prod_{i=1}^{j} f_{i}(x-ir)\,dr\big|^2\,dx\bigg|^{\half}
\end{align*}
by Cauchy-Schwarz. Note that $\| f_j\|_{L^2}\leq \|f_j\|_{L^{\infty}}$ so we are halfway done. Expanding the integral above and changing variables gives

\begin{align*}
 &     \int \big|\int   g(x,r)\prod_{i=1}^{j} f_{i}(x-ir)\,dr\big|^2\,dx \\=&
       \int   g(x,r)\bar{g}(x-u,r)\prod_{i=1}^{j} f_{i}(x-ir)\bar{f_{i}}(x-iu-ir)\,dr\,du\,dx\\=&
       \int   \triangle^1g(x,r;u)\prod_{i=1}^{j} \triangle^1 f_{i}(x-ir;iu)\,dr\,du\,dx\\=&
       \int   \triangle^1g(x-r,r;u)\prod_{i=1}^{j} \triangle^1_{iu} f_{i}(x-(i-1)r)\,dr\,du\,dx\\\leq&
              \int   \big|\Lambda_{j}(\triangle^1_u\tilde{g}(x-r,r;u);\triangle^1_{u} f_{j+1},\dots,\triangle^1_{ju}f_j\big|\,du
\end{align*}
and the claim is proved.

\end{proof}

\begin{proof}[Proof of Lemma \ref{thm:singlefrequencybound}]
 We apply Lemma \ref{thm:GCSstep} to obtain 
 \[
\big| \Lambda_{k+1}(g;h_0,\dots,h_{k-1},\mu_{n+1}-\mu_n)\big| \leq \|h_0\|_{L^{\infty}}\cdots\|h_{k-1}\|_{L^{\infty}}\big|\langle\mu_{n+1}-\mu_n, g\rangle_{U^k}\big|^{\frac{1}{2^k}}
\]
Set $\nu_n:=\mu_{n+1}-\mu_n$.
By Lemma \ref{thm:standard}, $\|h_0\|_{L^{\infty}}\cdots\|h_{k-1}\|_{L^{\infty}}\lesssim 2^{k(1-\alpha)n}$. On the other hand, one calculates 
\[\big|\langle\nu_n, g\rangle_{U^k}\big|= \big|  \sum_{\xi\in\Z^d,\eta\in{\Z^d}^k} \widehat{\triangle^kg}(0,\xi;\eta)\widehat{\triangle^k(\nu_n)}(0;\eta)\big|\]
Applying Holder's inequality  gives
\[\big|\langle\mu_n, g\rangle_{U^k}\big| \leq \|\widehat{\triangle^k g}(0,\xi;\eta)\|_{\ell^1_{\xi}\ell^p_{\eta}}\|\|\D{k}{\nu_n}{\cdot}\|_{p'}.
\]
By Proposition \identitiesA together with Lemma \ref{thm:standard}, 
\[\|\D{k}{\nu_n}{\cdot}\|_{2} = \|\nu_n\|_{U^{k+1}}^{2^k}\leq \|\nu_n\|_{L^{\infty}}^{2^k} \lesssim 2^{2^k(1-\alpha)n}.\]

 According to Lemma \ref{thm:start} and Proposition \ref{thm:rk},
 \[\|\D{k}{\nu_n}{\cdot}\|_{\infty}\lesssim \D{k}{\nu_n}{0} = \|\mu_{n+1}-\mu_n\|^{2^k}\lesssim 2^{-r_k n}.\]
 
 Applying Holder's inequality to interpolate between the $L^2$ and $L^{\infty}$ bounds gives
 
 \[\|\D{k}{\nu_n}{\cdot}\|_{p'}\lesssim  2^{-r_k(\beta)\frac{p'-2}{p'} n} 2^{2^k(1-\alpha)\frac{2}{p'}n} = 2^{-\omega_k^p(\beta)n} \]
where we have set 
\[\omega_k^p(\alpha,\beta):={-\left(r_k(\beta)\frac{p'-2}{p'} - 2^k(1-\alpha)\frac{2}{p'}\right)}.\]

\end{proof}

\begin{lemma}\label{thm:twofrequencybound}Suppose that $\mu$ is a measure on ${{{{\T}}}}$ satisfying \items  for $\alpha$ and $\beta$ sufficiently close to $d$. 
Then for any $g:{{{{\T}}}}\times{{{{\T}}}}\ra\C$, any approximate identity $(\phi_n)$, and any $m,N\in\N$, 

 \begin{align}\label{cauchy} \big| \Lambda_{k+1}(g;\phi_N\ast\mu)-\Lambda_{k+1}(g;\phi_m\ast\mu)\big| \lesssim \|\widehat{\triangle^k g}(0,\xi;\eta)\|_{\ell^1_{\xi}\ell^p_{\eta}}\sum_{n=m}^{N-1}2^{-\omega_k^p(\alpha,\beta)n}.\end{align}

\end{lemma}
\begin{proof}

We write the left hand side of (\ref{cauchy}) as a telescoping sum to obtain
\begin{align*}
 (\ref{cauchy})\leq& \sum_{n=m}^{N-1}\big|\Lambda_{k+1}(g;\phi_{n+1}\ast\mu)-\Lambda_{k+1}(g;\phi_n\ast\mu)\big|
\end{align*}
and applying Lemma \ref{thm:standard} we have
\begin{align}\label{telescoped}
 \big|\Lambda_{k+1}(g;\phi_N\ast\mu)-\Lambda_{k+1}(g;\phi_m\ast\mu)\big| \lesssim \sum_{n=m}^{N-1}\sup_{\vec{h}_n}\big|\Lambda_{k+1}(g;\vec{h}_n,\mu_n)\big|
\end{align}
where $\vec{h}_n\in \{\phi_{n+1}\ast\mu,\phi_{n}\ast\mu\}^k$ and $\mu_n:=\left(\phi_{n+1}-\phi_n\right)\ast\mu$.

Thus applying Lemma \ref{thm:singlefrequencybound} to (\ref{telescoped}), we have the bound 
\begin{align}\label{boundedtelescope}
 \big|\Lambda_{k+1}(g;\phi_N\ast\mu)-\Lambda_{k+1}(g;\phi_m\ast\mu)\big| \lesssim \sum_{n=m}^{N-1}2^{-\omega_k^p(\alpha,\beta)n}.
\end{align}
\end{proof}

\Ss{Support properties of the measure}

\begin{lemma}\label{thm:support}
 Suppose that $\mu$ is a measure on ${{{{\T}}}}$ satisfying \items  for $\alpha$ and $\beta$ sufficiently close to $d$.
 
Then 
\[\supp(\cap^{k+1}\mu)\subset \{(x,r) : x,x+r,\dots,x+kr\in \supp(\mu)\}\]
\end{lemma}
\begin{proof}
 Let $E=\{(x,r)\in{{{{\T}}}}\times{{{{\T}}}} : (x,x+r,\dots,x+kr)\notin \supp(\mu)^{k+1}\}$.

 Decompose $E = E^{(0)}\cup\cdots\cup E^{(k)}$, where $E^{(i)}:=\{(x,r) : x+ir\notin\supp(\mu)\}$.
 It suffices to show that $\cap^{k+1}\mu(E^{(i)}) = 0$ for each $i$.
 
 Fix $0\leq j \leq k$. Set
 
 \[E_m := \{(x,r)\in E^{(j)} : |x-jr-y|>\frac{1}{m} \forall y\in\supp(\mu)\} .\]
 
 It suffices to show that for every $m$, $\cap^{k+1}\mu(E_m)=0$.
 
 If we set $B_m = \{x : |x-y|>\frac{1}{m}\forall y\in \supp(\mu)\}$ we see that
 \[ E_m = \cup_{r\in{{{{\T}}}}} (B_m+jr)\times\{r\}.\]
 Thus $E_m$ is open.
  
 Since $E_m$ is open, we have
 \begin{align*}
 &\cap^{k+1}\mu(E_m)\leq \liminf \int 1_{E_m}(x,r) \prod_{i=0}^k\phi_n\ast\mu(x-ir)\,dx\,dr\\=&
 \liminf \prod_{0\leq i\neq j\leq k} \|\phi_n\ast\mu\|_{L^{\infty}} \int 1_{E_m}(x,r)\phi_n\ast\mu(x-jr)\,dx\,dr\\\leq&
2^{k(1-\alpha)n}\int1_{E_m}(x,r) 2^n\phi(2^n(x-jr-y))\,d\mu(y)\,dx\,dr
\end{align*}
using (\ref{thm:standard}). Recalling the decay hypothesis on $\phi$, we then have that the above is
\[ \lesssim 2^{k(1-\alpha)n} \int 1_{E_m})(x,r)2^n \left(2^n|x-jr-y|\right)^{-M}\,d\mu(y)\,dx\,dr\leq 2^{(1-\alpha)n} 2^{-(M-1)n} m^M
\]
using that $|x-jr-y|>\frac{1}{m}$ $\forall y\in\supp(\mu)$. 

Since $M>k(1-\alpha)$, this converges to $0$ as $n\ra\infty$, so the lemma is proved.
\end{proof}

\begin{lemma}\label{thm:nontri}
 Suppose that $\mu$ is a measure on $\T$ satisfying \items  for $\alpha$ and $\beta$ sufficiently close to $d$. 
Then 
\[\cap^{k+1}\mu({{{{\T}}}}\times\{0\}) = 0\]
 \end{lemma}

\begin{proof}
 We prove this by fairly direct calculation. Let $g:[-1,1]\ra[0,1]$ be any Schwarz function of compact support in $B(0,1)$ such that $g|_{B(0,\half)} \equiv 1$.
 
 Define $g_{\delta}(r):= g(\delta^{-1}r)$. Then $\cap^{k+1}\mu({{{{\T}}}}\times B(0,\delta/2))\leq \int g_{\delta}\,d\cap^{k+1}\mu$. So it suffices to show that
 \[\lim_{\delta\ra0} \int g_{\delta}\,d\cap^{k+1}\mu = 0.\]
 
 By Proposition \ref{thm:existenceofcount}, $\cap^{k+1}\mu$ exists, thus we may set $m=0$ and send $N\ra\infty$ in Lemma \ref{thm:twofrequencybound} to obtain
 \[\big|\int g_{\delta}\,\cap^{k+1}\mu-\int g_{\delta}(r)\prod_{i=0}^k \phi_0\ast\mu(x-ir)\,dx\,dr\big| \leq \|\widehat{\triangle^k g}(0,\xi;\eta)\|_{\ell^1_{\xi}\ell^p_{\eta}}\sum_{n=0}^{\infty}2^{-\omega_k^p(\alpha,\beta)n}.\]
 
 In other words, for sufficiently small $p>1$
 \[\big|\int g_{\delta}\,\cap^{k+1}\mu\big| \leq \|g_{\delta}\|_{L^1} + C \|\widehat{\triangle^k g}(0,\xi;\eta)\|_{\ell^1_{\xi}\ell^p_{\eta}}^{\frac{1}{2^k}}
 \]
 We are now done, since for any function $G$ on $[0,1]^d$ with $\hat{G}\in L^1({{{{\R}}}})$, for $p>1$ if $G_{\delta}(\vec{x}) = G(\delta^{-1}\vec{x})$ then applying Lemma \ref{thm:marcinkiewicz} and a change of variables,
\[\|\hat{G_{\delta}}\|_{\ell^p} \lesssim \|\hat{G_{\delta}}\|_{L^p}
= \delta^{2\frac{p-1}{p}}\|\hat{G}\|_{L^p}\]
and $G=\triangle^2 g(0;\cdot)$ is such a function.
 
\end{proof}

\begin{lemma}\label{thm:marcinkiewicz}
 Suppose that $f$ is a measure compactly supported in $[0,1]$. Then for any $p\in(1,\infty)$
 \[\sum_{\xi\in\Z}|\hat{f}(\xi)|^p \approx \int_{\R} |\hat{f}(\xi)|^p\,d\xi.\]
\end{lemma}
\begin{proof}
We first prove that 
 \[\sum_{\xi\in\Z}|\hat{f}(\xi)|^p \lesssim \int_{\R} |\hat{f}(\xi)|^p\,d\xi.\]

 Let $g$ be any Schwartz function equal to $1$ on the support of $f$ and for which $\supp\widehat{g}\subset[-\half,\half]$. 
 Then since $f = g f$, 
 \begin{align*}
  &|\hat{f}(\xi)| = |\hat{f}\ast\widehat{g}(\xi)| = |\int_{-\half}^{\half} \hat{f}(\xi-\eta)\widehat{g}(\eta)\,d\eta|\\\leq&
  \|\widehat{g}\|_{L^{\infty}}\int_{-\half}^{\half}|\hat{f}(\xi-\eta)|\,d\eta.
 \end{align*}
So $|\hat{f}(\xi)|\lesssim \int_{-\half}^{\half}|\hat{f}(\xi-\eta)|\,d\eta$. Thus by Holder's inequality
\begin{align*}
 	\sum_{\xi\in\Z}|\hat{f}(\xi)|^p \lesssim  \sum_{\xi\in\Z}\left(\int_{-\half}^{\half}|\hat{f}(\xi-\eta)|\,d\eta\right)^p\leq\sum_{\xi\in\Z}\int_{-\half}^{\half}|\hat{f}(\xi-\eta)|^p\,d\eta=\int_{\R}|\hat{f}(\xi)|^p\,d\xi.
\end{align*}

Now we prove the converse inequality. Fix a Schwartz function $h$ equal to $1$ on $\supp(f)$ such that $\supp(h)\subset[0,1]$. Then for any $\xi\in\R$
\[\hat{f}(\xi) = \widehat{fh}(\xi) = \sum_{n\in\Z} \hat{f}(n)\hat{h}(\xi-n)\]
by (111) of \cite{wolff}. By a variant of Young's inequality, we then have for any $1\leq p \leq \infty$ that $\|\hat{f}\|_{L^p} \leq \|\hat{f}\|_{\ell^p}\|\hat{h}\|_{L^{1}}$. 
Indeed,
\begin{align*}
 &\|\hat{f}\|_{L^1} = \int\big|\sum_{n\in\Z} \hat{f}(n)\hat{h}(\xi-n)\big| \leq \|\hat{f}\|_{\ell^1}\|\hat{h}\|_{L^1}\\
 &\|\hat{f}\|_{L^{\infty}} = \sup_{\xi\in\R} \big|\sum_{n\in\Z}\hat{f}(n)\hat{h}(\xi-n)\big|\lesssim\|\hat{f}\|_{\ell^{\infty}} \|\hat{h}\|_{L^1}
\end{align*}
where we have used the first part of the Lemma to bound $\sum_{n\in\Z} |\hat{h}(\xi-n)|$ by $\|\hat{h}\|_{L^1}$, and interpolating between the two gives for any $p\in[1,\infty]$
\[\|\hat{f}\|_{L^{p}}\lesssim \|\hat{f}\|_{\ell^p}.\]

\end{proof}

\Ss{Quantitative bounds for dense sets}

\begin{lemma}\label{thm:varn}
 Suppose $f\in C(\T^d)$ with $f\geq0$, $\int f \geq \delta>0$, and $\|f\|_{\infty}\leq M$, and fix $k+1$ distinct points $b_0,\dots,b_k\in\R^d$. 
 Then there is a constant $c(\delta,M)>0$ (depending also on the $k+1$-point configuration $b_0,\dots,b_k$)
 so that
\begin{align*}
 \Lambda^k(f) \geq c(\delta,M)
\end{align*}

\end{lemma}
\begin{proof}
A well-known result of Varnavides \cite{varn} stated originally for three-term progressions but equally valid for $k+1$-point configurations (Lemma \ref{thm:dvarn})
assures us the existence of a positive constant $c'(\delta,M)$ such that for $F:\Z_N\ra [0,M]$ with $\sum F\geq\delta$, 
$\Lambda^k(F):=\frac{1}{N^{2d}} \sum_{x\in[0,N-1]^d}\sum_{r\in[0,N-1]} \prod_{i=0}^{k} F(x-\lfloor rb_i\rfloor) \geq c'(\delta, M)>0$ for $N$ sufficiently large.

A discretizing procedure extends this to our setting. 

Namely, we note that it is sufficient to show the result for all $f$ in a subset $L^{\infty}(\T^d)$ whose closure contains $C^{\infty}_+$,
where the $+$ denotes non-negative functions, 
since then Dominated Convergence gives for any $\eps>0$ 
\begin{align*}
\Lambda^k(f) = \lim \Lambda^k(f_n)\geq c(\delta,(1+\eps)M)
\end{align*}

 With this in mind, suppose that $f\in L^{\infty}_+(\T^d)$ is constant on $[j_1(N)^{-1},(j_1+1)(N)^{-1})\times\cdots\times[j_d(N)^{-1},(j_d+1)(N)^{-1})$, 
 $j=(j_1,\dots,j_d)\in[0,N-1]^d$.
 Applying Varnavides Theorem for $k+1$-point configurations on $\Z_{}^d$ to $F({{j}}):= f(\frac{{{j}}}{N})$, we obtain the result.

In more detail,  let $K=2\lceil\sup\left\{1+|b_t|_{\infty} : 0\leq t\leq k\right\}\rceil$. Then

\begin{align}\notag{}
& \Lambda^k(f) = \iint \prod_{t=0}^k f(x-rb_t)= \sum_{\vec{i}\in[1,{KN}]^d}\sum_{{{j}}\in[1,{KN}]}\int_{[\frac{\vec{i}}{KN},\frac{\vec{i}+1}{KN}]} \int_{[\frac{{{j}}}{KN},\frac{{{j}}+1}{KN}]}  \prod_{t=0}^k f(x-rb_t)\\\geq&
\sum_{\vec{i}\in[1,{N}]^d}\sum_{{{j}}\in[1,{N}]} \int_{[\frac{K\vec{i}}{KN},\frac{K\vec{i}+1}{KN}]^d} \int_{[\frac{K{{j}}}{KN},\frac{K{{j}}+1}{KN}]}  \prod_{t=0}^k f(x-rb_t)\,dx\,dr\label{hms}
\end{align}
where for $\vec{i}=(i_1,\dots,i_d)\in \Z^d$ and $a \in\R$,, $\vec{i}+a =(i_1+a,\dots,i_d+a)$, and $[\vec{i},\vec{i}+a]=[i_1,i_1+a]\times\cdots\times[i_k,i_k+a]$.

By the choice of $K$ and the assumption that $f$ is constant on intervals of length $1/N$, as $x$ and $r$ vary in the above integral, $f(x-rb_t)$ remains constant 
(and equal to $F(\vec{i}-\lfloor{{j}}b_t\rfloor)$ by definition),
so that
\begin{align*}(\ref{hms})=&
\sum_{\vec{i}\in[1,N]^d}\sum_{{{j}}\in[1,{N}]} (KN)^{-2d} \prod_{t=0}^kF(\vec{i}-\lfloor{{j}}b_t\rfloor) 
\end{align*}

Since $\frac{1}{N}\sum_{{{j}}\in[1,N]^d} F({{j}}) = \int f\geq\delta$ and $\|F\|_{\infty} = \|f\|_{\infty}=M$, using the discrete Varnavides theorem, Lemma \ref{thm:dvarn}, applied to $F$ we can conclude
\begin{align*}
 \Lambda^k(f) = \frac{1}{4K^{2d}} \Lambda^k(F) \geq \frac{1}{4K^2} c'(\delta,M):=c(\delta,M).
\end{align*}
\end{proof}

\begin{lemma}\label{thm:dvarn}
 Suppose that the $k+1$ points $b_0,\dots,b_k\in\R^d$ are pairwise distinct.
 Then there is a $c(\delta,M)>0$ such that for any $F:\Z^d\ra[0,M]$ satisfying $\|F\|_1 \geq \delta>0$, we have
 
 \begin{align*}
\sum_{\vec{i}\in[1,N]^d}\sum_{{{j}}\in[1,{N}]} \prod_{t=0}^kF(\vec{i}-\lfloor {{j}}b_t\rfloor) \geq c(\delta,M)
\end{align*}

\end{lemma}
\begin{proof}

 The proof is a direct adaptation of Varnavides original argument \cite{varn}, with no change save that of notation. We include it here only for completeness.
 
 It is easy to see that it suffices to show the result where $F=1_A$ for some set $A\subset\Z^d$, so we show it in this case.
 
 By Lemma \ref{thm:szem}, 
 
  \begin{align*}
\sum_{\vec{i}\in[\vec{1},{M}]^d}\sum_{{{j}}\in[1,{M}]^d} \prod_{t=0}^k1_A(\vec{i}-\lfloor{{j}}b_t\rfloor) >0
\end{align*}
whenever $|A|\geq \frac{\delta}{2} M^d$

Choose $N$ large. Let $A\subset [1,N]^d$ with $|A|\geq \delta N^d$. Fix $\vec{l}\in[1,N/M]^d$.
                
             Partition $[1,N]^d$ into higher-dimensional progressions of the form
             \begin{align}\mathcal{P}=\{ \vec{n}=(n_1,\dots,n_d) : n_i = m_i+j l_i, 1\leq i\leq d, 0\leq j\leq M'\}\label{PP}\end{align}
             where $\vec{m}=(m_1,\dots,m_d)\in[1,N]^d$, and $M\leq M'\leq 2M$. This is possible via a greedy algorithm argument.
             
             Note that our partition consists of between $N^d/(2M)^d$ and $N^d/M^d$ of these progressions.

 Call such a progression $\mathcal{P}$ \textit{good} if $|A\cap \mathcal{P}|/ |\mathcal{P}| \geq \delta/2$. 
 
 Pigeonholing gives us that there are at least $\frac{\delta}{2} N^d$ elements of $A$ on good progressions.
 
 Consequently, there are at least $\frac{\delta}{2}N^d$ elements of $A$ which lie on good progressions for our fixed choice of $\vec{l}$,
 and so at least 
 \[\frac{N^d}{M^d}\frac{\delta}{2}\frac{N^d}{(2M)^d}=\frac{N^{2d}}{2^{d+1}M^{2d}}\delta\]
good progressions.
 
 The set $A$ restricted to each such progression has a density on it of at least $\delta/2$, and each such progression is of length at least $M$, so we may invoke the hypothesis that
 multiple recurrence holds for the $b_i$
 to conclude the that this portion of $A$ contains $x-[nb_0],\dots,x-[nb_k]$ for some $x\in\Z^d$ and $n\in\Z$.
 
 To conclude a count of the number of distinct such point patterns in $A$, we must determine how many progressions may have simultaneously contained a given collection 
 $x-[nb_0],\dots,x-[nb_k]$.
 
 A quick upper bound of $(2M)^{3d}$ is available by positing that if a progression has the form (\ref{PP}) and contains a point $x$, then there are at most $2^dM^d$ choices of 
 $\vec{m}$ since the length of the progression is at most $2M$, there are at most $2^dM^d$ choices for the components of $\vec{m}+M'\vec{l}$ for the same reason, 
 and after specifying $\vec{m}$ and $\vec{m}+M'\vec{l}$, the progression is determined once we specify $\vec{l}$; if $x-[nb_0],\dots,x-[nb_k]\in\mathcal{P}$, then $[nb_i]-[nb_j]$ must
 divided by $\vec{l}$ in the sense that for each $1\leq t\leq d$, $ l_t | [nb_{i,t}]-[nb_{j,t}]$; for $D_t:=\max_{i,j}[nb_{i,t}]-[nb_{j,t}]$,
 write $D_t = c_t l_t$. In order that $x-[nb_0],\dots,x-[nb_k]$ fit in $\mathcal{P}$, it is necessary that
 $Ml_t = MD_t/c_t\geq kD_t$ so that $c_t\leq  M/k$. Since this holds for each $1\leq t\leq d$, there are fewer than $M^d$ choices for $\vec{l}$ which will result in a progression which contains
 $x+[nb_0],\dots,x+[nb_k]$. This gives a total of at most 
 \[(2M)^{3d} \]
 possible progressions $\mathcal{P}$ which overlap on the configuration $x+[nb_0],\dots,x+[nb_k]$, 
and so the total number of images of the form $x+[nb_0],\dots,x+[nb_k]$ in $A$ is at least one such image per every good progression, divided by the number of times the same image is contained in a different good progression,
which is at least 
\[\left(\frac{N^{2d}}{2^{d+1}M^{2d}}\delta\right)/\left(2M\right)^{3d} = C(M)N^{2d}\]
such images.

 \end{proof}

Define a trivial affine image of the point set $b_0,\dots,b_k$ to be an image $x_0,\dots,x_k$ for which $x_0=x_1=\cdots=x_k$.

 \begin{lemma}\label{thm:szem}
  Suppose that the $k+1$ points $b_0,\dots,b_k\in\R^d$ are pairwise distinct. Then for any $\delta>0$, there exists an $N\in\Z$ such that if $A\subset \Z_N^d$ with $\frac{|A|}{N^d}>\delta$,
  then $A$ contains a non-trivial affine image $a+\lfloor tb_0 \rfloor,\dots, a+\lfloor tb_k\rfloor$.
   \end{lemma}
\begin{proof}

 We use hypergraph techniques originated by Solymosi \cite{solymosi}. We follow the approach as outlined in \cite{cookmagyartatchai}. 
 
 By affine-invariance of the statement of the lemma, we may suppose that $\{b_0,\dots,b_k\} = \{0,x_1,\dots,x_k\}$ for some $x_i\in\R^d$. 
 
 We complete the pattern into a simplex by setting 
 \[E = \{0,(x_1,y_1),\dots,(x_k,y_k),z_{k+1},\dots,z_{k+d}\}\subset \R^{d+k}
 \]
 for some linearly independent $\{y_1,\dots,y_k\}\subset\R^{d+k}$
 and some $z_{k+1},\dots,z_{k+d}$ such that $E\setminus\{0\}$ forms a basis for $\R^{d+k}$. By showing that $A':=A\times[N]^k\subset[N]^{d+k}$ contains an affine image of $E$ 
 for $M$ sufficiently large depending on $\delta$, we will have found that also then $A$ contains an affine image of $\{0,x_1,\dots,x_k\}$. 
 
 We parametrize affine images of $E$ by points $(c_1,\dots,c_{d+1})$. 
 Given an $e\subset[d+1]$ of size $d$, define $p_e$ to be the solution to the system $\{p_e\cdot n_i = c_i\}_{i\in e}$. Then for any choice of $(c_1,\dots,c_{d+1})$, the collection $\{p_e\}_{e\subset[d+1],|e|=d}$
 is an affine image of $E$.
 
 Set $c_e =(c_i)_{i\in e}\in V_e:= \prod_{j\in e} V_j$ and define $L_e^j(c_e):=[p_e]_j$, $\mathcal{L}_{e}=\{L_e^j\}_{j=1}^d$. 
 
 We now define a hypergraph $\mathcal{H}$ with vertex sets $V_1,\dots,V_{d+1}$ with $V_j = [-\frac{N}{2},\frac{N}{2}]$ and whose edges are elements
 $c_e \in V_e$, $e\subset[d+1]$ with $|e|=d$ for which $\lfloor\mathcal{L}_e(c_e)\rfloor\in A + (N\Z)^d$. A simplex is then a collection 
 \[c = (c_1,\dots,c_{d+1})\in \prod_{j\in[d+1]}V_j.\]
 
 Suppose now that $A$ contains no non-trivial affine images of $E$. Then $A$ contains at most $|A| = \delta N^d = \frac{\delta}{N} N^{d+1}$ affine images of $E$, so by Gowers' hypergraph removal lemma 
 (for ease of reference, included as Theorem \ref{thm:hyprem}) below), there is an $\eps=\eps(\frac{\delta}{N})$ so that one can remove at most $\eps N^{d}$ edges from $\mathcal{H}$ so that $\mathcal{H}$ is simplex free. For 
 $N$ sufficiently large, $\eps$ can be chosen to be less than $\delta$. But each such edge is the edge of at most one trivial simplex, so one must remove at least $|A| = \delta N^d$ 
 edges to removal them all. This contradiction guarantees the existence of at least one nontrivial affine image of $E$.

\end{proof}

\begin{thm}[ Hypergraph Removal Lemma (\cite{gowersHypergraph},Theorem {10.1})]\label{thm:hyprem}. Let $k$ be a positive integer. Then for
every $a>0$ there exists $c>0$ with the following property. Let 
$H$ be a $(k+1)$-partite $k$-uniform hypergraph with vertex sets 
$\seq X {{k+1}}$, and let $N_i$ be the size of $X_i$. Suppose that
$H$ contains at most $c\prod_{i=1}^{k+1}N_i$ simplices. Then for 
each $i\le k+1$ one can remove at most $a\prod_{j\ne i}N_j$ 
edges of $H$ from $\prod_{j\ne i}X_j$ in such a way that after 
the removals one is left with a hypergraph that is simplex-free.
\end{thm}

\section{Further work}
There are two directions which immediately suggest themselves. 

The first is to extend the results of the present paper to multidimensional configurations. Along this direction, everything
goes through up to Lemma \ref{thm:GCSstep}, which fails to hold for example, in the case of ``corners'' $\{\vec{x},\vec{x}+re_1,\dots,\vec{x}+re_d\}$. In the linear case, the results
of \cite{chanlabapramanik} obtain corners in $\R^2$ (those controlled by Fouier analysis) and obtain some other configurations, but leaves open whether all configurations controlled
by the linear Fourier transform are obtained (to the cognoscenti, \cite{chanlabapramanik} does not obtain all patterns of Cauchy-Schwarz complexity $1$). 
It would be desireable to to understand better the situation as regards both of these points.

Secondly, perhaps the next natural step would be to obtain an analogue of Bergelson and Leibman's polynomial ergodic theorem, \cite{bergelsonleibman}.

\bibliographystyle{plain}

\bibliography{biblio.bib}

\begin{thebibliography}{10}

\bibitem{cookmagyartatchai}
T.~Titichetrakun B.~Cook, A.~Magyar.
\newblock A multidimensional {S}zemer\'edi theorem in the primes.

\bibitem{bergelsonleibman}
V.~Bergelson and A.~Leibman.
\newblock Polynomial extensions of van der {W}aerden's and {S}zemer\'edi's
  theorems.
\newblock {\em J. Amer. Math. Soc.}, 9(3):725--753, 1996.

\bibitem{M1}
M.~Carnovale.
\newblock Gowers norms for singular measures.
\newblock 2013.

\bibitem{M2}
M.~Carnovale.
\newblock Higher order fourier dimension and frequency decompositions.
\newblock 2013.

\bibitem{gowers}
W.~T. Gowers.
\newblock A new proof of {S}zemer\'edi's theorem.
\newblock {\em Geom. Funct. Anal.}, 11(3):465--588, 2001.

\bibitem{gowersHypergraph}
W.~T. Gowers.
\newblock Hypergraph regularity and the multidimensional {S}zemer\'edi theorem.
\newblock {\em Ann. of Math. (2)}, 166(3):897--946, 2007.

\bibitem{Keleti}
Tam{\'a}s Keleti.
\newblock Construction of one-dimensional subsets of the reals not containing
  similar copies of given patterns.
\newblock {\em Anal. PDE}, 1(1):29--33, 2008.

\bibitem{Laba}
Izabella {\L}aba and Malabika Pramanik.
\newblock Arithmetic progressions in sets of fractional dimension.
\newblock {\em Geom. Funct. Anal.}, 19(2):429--456, 2009.

\bibitem{Labadiff}
Izabella {\L}aba and Malabika Pramanik.
\newblock Maximal operators and differentiation theorems for sparse sets.
\newblock {\em Duke Math. J.}, 158(3):347--411, 2011.

\bibitem{Maga}
Peter Maga.
\newblock Full dimensional sets without given patterns.
\newblock {\em Real Anal. Exchange}, 36:79--90, 2010.

\bibitem{Mitsis}
Themis Mitsis.
\newblock A {S}tein-{T}omas restriction theorem for general measures.
\newblock {\em Publ. Math. Debrecen}, 60(1-2):89--99, 2002.

\bibitem{Mocken}
Gerd Mockenhaupt.
\newblock Salem sets and restriction properties of fourier transforms.
\newblock {\em Geom. Funct. Anal.}, (6):1579--1587, 1996.

\bibitem{solymosi}
J{\'o}zsef Solymosi.
\newblock Note on a generalization of {R}oth's theorem.
\newblock In {\em Discrete and computational geometry}, volume~25 of {\em
  Algorithms Combin.}, pages 825--827. Springer, Berlin, 2003.

\bibitem{chanlabapramanik}
M.~Pramanik V.~Chan, I.~{\L}aba.
\newblock Finite configurations in sparse sets.
\newblock {\em J. Anal. Math.}

\bibitem{varn}
P.~Varnavides.
\newblock On certain sets of positive density.
\newblock {\em J. London Math. Soc.}, 34:358--360, 1959.

\bibitem{wolff}
Thomas~H. Wolff.
\newblock {\em Lectures on harmonic analysis}, volume~29 of {\em University
  Lecture Series}.
\newblock American Mathematical Society, Providence, RI, 2003.
\newblock With a foreword by Charles Fefferman and preface by Izabella {\L}aba,
  Edited by {\L}aba and Carol Shubin.

\end{thebibliography}

%
%
%

\vskip0.5in

\noindent Marc Carnovale
\\ The Ohio State University \\ 231 W. 18th Ave\\ Columbus Oh, 43210 United States\\ {\em{Email: carnovale.2@osu.edu}}

\end{document}